\documentclass[a4paper,12pt]{amsart}

\usepackage[T1]{fontenc}
\usepackage[utf8]{inputenc}
\usepackage{lmodern}
\usepackage{amsfonts,amsmath,amssymb,mathrsfs}
\usepackage{enumitem}
\usepackage{microtype}
\usepackage[colorlinks=true,linkcolor=blue,citecolor=blue,urlcolor=blue]{hyperref}

\allowdisplaybreaks

\numberwithin{equation}{section}

\theoremstyle{plain}
\newtheorem{thm}{Theorem}[section]
\newtheorem{prop}[thm]{Proposition}
\newtheorem{cor}[thm]{Corollary}
\newtheorem{lemma}[thm]{Lemma}

\theoremstyle{definition}
\newtheorem{defi}[thm]{Definition}

\newtheorem{rem}[thm]{Remark}
\newtheorem{ex}[thm]{Example}

\newcommand{\R}{\mathbb R}
\newcommand{\Rp}{(0,\infty)}
\newcommand{\N}{\mathbb N}
\newcommand{\Q}{\mathcal Q}
\newcommand{\Qt}{\mathcal Q^{*}}

\newcommand{\one}{\mathbf 1}

\newcommand{\supp}{\operatorname{supp}}

\hypersetup{
  pdftitle={Product-weighted Pitt inequalities for mixed-tail potentials},
  pdfauthor={Niyaz Tokmagambetov},
  pdfkeywords={Pitt inequality, Fourier transform, mixed norm, product weight, mixed derivative, general monotonicity, anisotropic inequality}
}

\begin{document}

\title[Product-weighted Pitt inequalities]{Product-weighted Pitt inequalities for mixed-tail potentials}

\author[N. Tokmagambetov]{Niyaz Tokmagambetov}
\address{Institute of Mathematics and Mathematical Modeling, 28 Shevchenko St.,
050010 Almaty, Kazakhstan}
\email{niyaz.tokmagambetov@gmail.com}
\email{tokmagambetov@math.kz}

\date{November 2026}
\keywords{Pitt inequality; Fourier transform; mixed norm; product weight; mixed derivative; general monotonicity; anisotropic inequality}
\subjclass[2020]{42B10, 42B35, 26D10, 46E30}

\begin{abstract}
Motivated by the one-dimensional and radial theory of general monotone functions, we introduce a natural multiparameter general-monotonicity class for product Fourier inequalities.  The monotonicity condition is replaced by an intrinsic mixed-tail condition: a function on the positive orthant is the upper-tail potential of a complex Radon measure, and the integrated total variation of this measure is controlled by a local integral of the function.  The representing measure is recovered as the full mixed distributional derivative, so the condition is intrinsic rather than coordinatewise.  For coordinatewise even extensions we prove the anisotropic mixed-norm estimate
\[
 \|\Pi_{\alpha}\widehat f_{A}\|_{L_{\vec q}(\mathbb R^d)}
 \le C\|\Pi_{\beta}f\|_{L_{\vec p}(\mathbb R^d)},
 \qquad
 \beta_k=1-\frac1{p_k}-\frac1{q_k}-\alpha_k,
\]
where $1\le p_k\le q_k<\infty$ and $-1/q_k<\alpha_k<1-1/q_k$.  Here $\widehat f_A$ is an Abel-summed Fourier transform and agrees almost everywhere with the ordinary Fourier transform for $f\in L^1$.  The exponent relation is forced by independent coordinate dilations.  Within the Hardy range $\alpha_k>-1/q_k$, finite input norm for any nonzero function in the class forces the upper restrictions $\alpha_k<1-1/q_k$.  The hypotheses are formulated in the spirit of general monotonicity: mixed variation is controlled by rectangular integrals, and no tensor-product structure is imposed.  The class contains all sectorial mixed-tail potentials, an infinite-dimensional cone generated by non-product positive measures satisfying the mixed-moment condition, and real nonseparable functions which are not monotone in the coordinate variables.  Nonnegative compactly supported model families show sharpness of both endpoints.
The atomic construction also gives a genuinely sign-changing real family.
\end{abstract}

\maketitle

\section{Introduction}

Throughout, $d\in\N$.  For $f\in L^1(\R^d)$ we use the Fourier transform
\[
 \widehat f(\xi)=\int_{\R^d}e^{i x\cdot\xi}f(x)\,dx.
\]
A product-weighted Pitt inequality compares a weighted norm of $\widehat f$ with a weighted norm of $f$.  In its anisotropic form the weights and the Lebesgue exponents may be different in the coordinate directions.  Such estimates are natural for product geometries and are closely related to rectangular Hardy operators.

Classical two-weight Fourier inequalities were developed in, among other works, \cite{Hei84,JS84,Muc83a,Muc83b,BH92,BH03}.  Product-weighted Hardy--Littlewood and Pitt inequalities, including Hardy--Ces\`aro and Hardy--Bellman transforms, were studied recently in \cite{DNTW23}; for a recent Pitt-type development involving Dirichlet kernels, see \cite{NTW26}.  A closely related line of work asks how far the monotonicity hypotheses in Fourier inequalities can be weakened.  In one dimension and in radial settings, this led to the theory of general monotone functions and to refinements of the Boas--Pitt theory \cite{LT08,LT09,LT11,GLT11,DGT13}.  Anisotropic Fourier integrability problems for coordinatewise monotone functions in Lebesgue and Lorentz spaces were studied in \cite{Muk15,Muk17}.

The present paper is a product-space continuation of that general-monotone philosophy.  The aim is not to change the coordinatewise scaling law, which is already forced by dilation.  Rather, the aim is to replace coordinatewise monotonicity by a natural intrinsic condition that still supplies the cancellation and variation control needed in the Fourier estimate.  The condition should allow nonseparable functions, signed or sectorial mixed derivatives, and arbitrary coupling between coordinates.  The object below is designed for this purpose: the full mixed derivative is required to be a Radon measure, and its upper-tail total variation is controlled by a rectangular average of the function itself.

Let $g$ be a function on $\Rp^d$.  The basic object is its full mixed derivative
\[
 \mu_g=(-1)^dD_1\cdots D_d g.
\]
We require $\mu_g$ to be a complex Radon measure and $g$ to be its upper-tail potential,
\[
 g(x)=\mu_g\bigl([x_1,\infty)\times\cdots\times[x_d,\infty)\bigr).
\]
The mixed total variation above $x$ is then
\[
 W_g(x)=|\mu_g|\bigl([x_1,\infty)\times\cdots\times[x_d,\infty)\bigr).
\]
Our rectangular mixed general-monotonicity condition is the scale-invariant estimate
\begin{equation}\label{eq:intro-rmgm}
 \int_{\Q(r)}W_g(x)\,dx
 \le K\int_{\Q(\lambda r)}|g(x)|\,dx,
 \qquad r\in\Rp^d,
\end{equation}
where $\Q(r)=\prod_k(0,r_k)$ and $\lambda r$ is componentwise multiplication.  For a positive measure $\mu_g$, equality holds in \eqref{eq:intro-rmgm} with $K=1$ and $\lambda=\one$.  More generally, the condition holds for sectorial complex measures and for signed measures with controlled cancellation.

The mixed derivative provides the product cancellation that ordinary first-order bounded variation does not.  Consider the coordinatewise even extension
\[
 f(x)=g(|x_1|,\ldots,|x_d|).
\]
Repeated Fubini integration gives
\begin{equation}\label{eq:intro-transform-formula}
 \widehat f_A(\xi)
 =2^d\int_{\Rp^d}\prod_{k=1}^d\frac{\sin(t_k\xi_k)}{\xi_k}\,d\mu_g(t),
\end{equation}
for nonzero coordinate frequencies.  Consequently,
\[
 |\widehat f_A(\xi)|
 \le 2^d\int_{\Rp^d}\prod_{k=1}^d\min\{t_k,|\xi_k|^{-1}\}\,d|\mu_g|(t).
\]
Condition \eqref{eq:intro-rmgm} turns the right-hand side into a rectangular Hardy average of $|g|$.  A mixed-norm tensorization of the one-dimensional power Hardy inequality then gives the main theorem.

The resulting exponent relation is
\begin{equation}\label{eq:intro-balance}
 \beta_k=1-\frac1{p_k}-\frac1{q_k}-\alpha_k,
 \qquad k=1,\ldots,d.
\end{equation}
It is exactly the relation forced by independent dilations in each coordinate.  Thus the new class has the same coordinatewise dilation balance as the known anisotropic monotone theorems, while including genuinely nonseparable and nonmonotone examples.

\subsection*{Proof architecture and standard tools}
The proof separates four independent ingredients.  First, the mixed-tail representation identifies the full mixed distributional derivative with a Radon measure and gives the integrated variation identity.  Second, Abel summation and repeated Fubini integration convert the coordinatewise even Fourier transform into a product sine-kernel integral against this measure.  Third, rectangular mixed general monotonicity converts the resulting total-variation bound into a rectangular Hardy average of $|g|$.  Fourth, the anisotropic estimate follows by tensorizing the one-dimensional power Hardy inequality through Minkowski's integral inequality.

This architecture is intended to be a direct multiparameter analogue of the general-monotone method.  In the one-dimensional theory, monotonicity can be replaced by local control of variation.  Here the full mixed derivative plays the role of the variation measure, and the rectangular condition \eqref{eq:intro-rmgm} is the product analogue of that local control.  The point is that cancellation is measured in all coordinates simultaneously, so the class is not forced to be a product of one-dimensional general-monotone functions.

\begin{rem}[Standard analytic tools]
We use the standard forms of Fubini--Tonelli, dominated convergence, Young's convolution inequality on $\mathbb R$, Minkowski's integral inequality, mixed-norm tensorization, Radon-measure total variation and polar decomposition, the distributional derivative calculus for Radon measures on the open orthant, Abel dominated convergence for the sine-kernel formula, agreement of Abel summation with the ordinary Fourier transform in the $L^1$ case, and the elementary support identity for tensor products of nonzero distributions.  These are background analytic tools.  The new content is the mixed-tail general-monotonicity condition and its use in the dilation-balanced Pitt estimate.
\end{rem}

\section{Mixed norms and an anisotropic rectangular Hardy inequality}\label{sec:hardy}

For $\rho=(\rho_1,\ldots,\rho_d)\in\R^d$ and $x\in\R^d$ with
$x_1\cdots x_d\ne0$, put
\[
 \Pi_\rho(x)=\prod_{k=1}^d|x_k|^{\rho_k}.
\]
On the coordinate hyperplanes we assign any fixed measurable value, for instance
$\Pi_\rho(x)=0$.  This convention removes the ambiguities caused by negative
exponents and by $0^0$ and has no effect on any mixed norm.
For $r\in\Rp^d$ define
\[
 \Q(r)=\prod_{k=1}^d(0,r_k),
 \qquad
 \Qt(r)=\prod_{k=1}^d[r_k,\infty).
\]
All vector inequalities and products are componentwise.

Let $\vec s=(s_1,\ldots,s_d)$ with $1\le s_k<\infty$.  The mixed norm is
\begin{align}\label{eq:mixed-norm}
 \|F\|_{L_{\vec s}(\Rp^d)}
 :=\left(\int_0^\infty\cdots
 \left(\int_0^\infty |F(x_1,\ldots,x_d)|^{s_1}\,dx_1\right)^{s_2/s_1}
 \cdots dx_d\right)^{1/s_d}.
\end{align}
The same convention is used on $\R^d$.  We call a measurable function
$F$ on $\Rp^d$ a tensor product if
$F(x)=\prod_{k=1}^d F_k(x_k)$ almost everywhere for measurable one-dimensional
factors $F_k$; otherwise $F$ is called nonseparable.  Whenever such an $F$ is
nonzero and belongs to $L^1_{\mathrm{loc}}(\Rp^d)$, the factors may be taken
nonzero and in $L^1_{\mathrm{loc}}(0,\infty)$.  Indeed, for each of the other
factors one may choose a bounded measurable set on which its absolute integral
is finite and positive; Tonelli's theorem on the resulting product set then
forces local integrability of the remaining factor.

We first recall the one-dimensional power Hardy inequality in the precise form needed below.

\begin{lemma}[Power Hardy inequality]\label{lem:power-hardy}
Let $1\le p\le q<\infty$, let $A<-1$, and set
\[
 B=\frac pq(A+1)+p-1.
\]
Then every nonnegative measurable $h$ on $(0,\infty)$ satisfies
\begin{equation}\label{eq:power-hardy}
 \left(\int_0^\infty r^A\left[\int_0^r h(t)\,dt\right]^qdr\right)^{1/q}
 \le C
 \left(\int_0^\infty r^B h(r)^p\,dr\right)^{1/p}.
\end{equation}
\end{lemma}

\begin{proof}
Put $r=e^s$, $t=e^u$, and
\[
 \delta=-\frac{A+1}{q}>0,
 \qquad
 H(u)=e^{(B+1)u/p}h(e^u).
\]
The balance relation gives $(B+1)/p=1-\delta$.  Hence
\[
 e^{(A+1)s/q}\int_0^{e^s}h(t)\,dt
 =\int_{-\infty}^s e^{-\delta(s-u)}H(u)\,du.
\]
The right-hand side is convolution on $\R$ with
$k(v)=e^{-\delta v}\mathbf1_{(0,\infty)}(v)$.  Choose $r_0\in[1,\infty)$ so that
$1+1/q=1/p+1/r_0$.  Since $1\le p\le q<\infty$, such an exponent satisfies
$1\le r_0<\infty$; when $p=1$, it is $r_0=q$.  Young's convolution inequality,
including its $L^1$ endpoint, applies because $k\in L^{r_0}(\R)$.  Undoing the
logarithmic substitution proves \eqref{eq:power-hardy}.
\end{proof}

\begin{lemma}[Inverted Hardy operator]\label{lem:inverted-hardy}
Let $1\le p\le q<\infty$, $\lambda>0$, and $\alpha>-1/q$.  Define
\begin{equation}\label{eq:one-dimensional-beta}
 \beta=1-\frac1p-\frac1q-\alpha.
\end{equation}
Then
\begin{equation}\label{eq:inverted-hardy}
 \left(\int_0^\infty
 \left[\xi^\alpha\int_0^{\lambda/\xi}h(x)\,dx\right]^q d\xi\right)^{1/q}
 \le C
 \left(\int_0^\infty [x^\beta h(x)]^pdx\right)^{1/p}
\end{equation}
for every nonnegative measurable $h$.
\end{lemma}

\begin{proof}
Set $r=\lambda/\xi$.  The $q$th power of the left-hand side is
\[
 \lambda^{\alpha q+1}
 \int_0^\infty r^{-\alpha q-2}
 \left[\int_0^r h(x)\,dx\right]^qdr.
\]
Apply Lemma~\ref{lem:power-hardy} with $A=-\alpha q-2<-1$.  The corresponding input exponent is
\[
 B=\frac pq(-\alpha q-1)+p-1
 =p\left(1-\frac1p-\frac1q-\alpha\right)=p\beta.
\]
\end{proof}

\begin{prop}[Anisotropic mixed-norm Hardy inequality]\label{prop:mixed-hardy}
Let
\[
 1\le p_k\le q_k<\infty,
 \qquad
 \alpha_k>-\frac1{q_k},
 \qquad k=1,\ldots,d,
\]
and let $\lambda\in(0,\infty)^d$.  Put
\begin{equation}\label{eq:mixed-beta}
 \beta_k=1-\frac1{p_k}-\frac1{q_k}-\alpha_k.
\end{equation}
Then every nonnegative measurable $h$ on $\Rp^d$ satisfies
\begin{equation}\label{eq:mixed-hardy}
 \left\|
 \Pi_\alpha(\xi)
 \int_{\Q(\lambda/\xi)}h(x)\,dx
 \right\|_{L_{\vec q}(\Rp^d)}
 \le C
 \|\Pi_\beta h\|_{L_{\vec p}(\Rp^d)},
\end{equation}
where $\lambda/\xi=(\lambda_1/\xi_1,\ldots,\lambda_d/\xi_d)$.
\end{prop}

\begin{proof}
Set $H(x)=\Pi_\beta(x)h(x)$.  For each coordinate define the positive integral operator
\[
 (T_k\varphi)(\xi_k)
 =\xi_k^{\alpha_k}
 \int_0^{\lambda_k/\xi_k}x_k^{-\beta_k}\varphi(x_k)\,dx_k.
\]
Lemma~\ref{lem:inverted-hardy} gives
$T_k:L^{p_k}(0,\infty)\to L^{q_k}(0,\infty)$.
Moreover, the left-hand side of \eqref{eq:mixed-hardy} is the mixed norm of
$T_1\cdots T_dH$.

We tensorize in the order of the mixed norm.  For $d=2$, the estimate for $T_1$ first gives
\[
 \|T_1T_2H\|_{L^{q_2}_{\xi_2}(L^{q_1}_{\xi_1})}
 \le C_1\|T_2H\|_{L^{q_2}_{\xi_2}(L^{p_1}_{x_1})}.
\]
Since $T_2$ is positive, Minkowski's integral inequality yields, for every $\xi_2$,
\[
 \|T_2H(\cdot,\xi_2)\|_{L^{p_1}_{x_1}}
 \le T_2\bigl(\|H(\cdot,x_2)\|_{L^{p_1}_{x_1}}\bigr)(\xi_2).
\]
The scalar $L^{p_2}\to L^{q_2}$ estimate then gives the desired two-variable bound.  The same argument applies inductively: after the first $j-1$ scalar estimates, Minkowski moves the positive operator $T_j$ outside the Banach mixed norm $L^{p_{j-1}}\cdots L^{p_1}$, and the scalar $L^{p_j}\to L^{q_j}$ estimate is applied.  This proves \eqref{eq:mixed-hardy}.
\end{proof}

\begin{rem}[Tensorization convention]\label{rem:tensorization-convention}
The proof uses only positivity of the one-dimensional operators and the order in which the mixed norm is taken.  Thus no interpolation of mixed-norm spaces is hidden in Proposition~\ref{prop:mixed-hardy}; the estimate is obtained by iterating scalar Hardy bounds and moving the remaining positive integral operators through the already-estimated variables by Minkowski's inequality.  The argument includes the endpoint $p_k=1$, since both Young's convolution inequality and Minkowski's integral inequality remain valid in $L^1$.
\end{rem}

\section{Mixed-tail potentials and the Abel Fourier transform}\label{sec:mixed-class}

\begin{defi}[Upper-tail potential]\label{def:tail-potential}
A measurable function $g:\Rp^d\to\mathbb C$ is an upper-tail potential if there is a complex Radon measure $\mu_g$ on $\Rp^d$ such that
\begin{equation}\label{eq:tail-potential}
 g(x)=\mu_g(\Qt(x))
 \quad\text{for almost every }x\in\Rp^d,
\end{equation}
and
\begin{equation}\label{eq:mixed-moment-finite}
 \mathcal M_g(r)
 :=\int_{\Rp^d}\prod_{k=1}^d\min\{r_k,t_k\}\,d|\mu_g|(t)<\infty
 \quad\text{for every }r\in\Rp^d.
\end{equation}
\end{defi}

Condition \eqref{eq:mixed-moment-finite} implies $|\mu_g|(\Qt(x))<\infty$ for every $x\in\Rp^d$, because
\[
 \mathcal M_g(x)\ge
 \left(\prod_{k=1}^d x_k\right)|\mu_g|(\Qt(x)).
\]
Thus the tail in \eqref{eq:tail-potential} is well defined.  Throughout the paper, an upper-tail potential is represented by its canonical pointwise tail version
\begin{equation}\label{eq:canonical-tail-version}
 g^{\sharp}(x):=\mu_g(\Qt(x)),
 \qquad x\in\Rp^d,
\end{equation}
and we henceforth write $g$ for $g^{\sharp}$.  This version agrees almost everywhere with the original measurable function in \eqref{eq:tail-potential}; changing to it does not affect any local integral or mixed norm.  Lemma~\ref{lem:intrinsic-mixed-derivative} shows that the representing measure, and hence the canonical version, is uniquely determined by the almost-everywhere class of $g$.

The following identity explains the form of \eqref{eq:mixed-moment-finite}.

\begin{lemma}[Integrated mixed variation]\label{lem:integrated-variation}
For every upper-tail potential and every $r\in\Rp^d$,
\begin{equation}\label{eq:integrated-variation}
 \mathcal M_g(r)
 =\int_{\Q(r)}|\mu_g|(\Qt(x))\,dx.
\end{equation}
In particular, $g\in L^1_{\mathrm{loc}}(\Rp^d)$ and
\begin{equation}\label{eq:g-below-variation}
 \int_{\Q(r)}|g(x)|\,dx\le \mathcal M_g(r).
\end{equation}
\end{lemma}

\begin{proof}
Fubini's theorem gives
\begin{align*}
 \int_{\Q(r)}|\mu_g|(\Qt(x))\,dx
 &=\int_{\Rp^d}
 \left|\{x\in\Q(r):x_k\le t_k\text{ for all }k\}\right|\,d|\mu_g|(t)\\
 &=\int_{\Rp^d}\prod_{k=1}^d\min\{r_k,t_k\}\,d|\mu_g|(t).
\end{align*}
The second assertion follows from $|g(x)|\le|\mu_g|(\Qt(x))$.
\end{proof}

\begin{lemma}[Intrinsic mixed derivative]\label{lem:intrinsic-mixed-derivative}
Let $g$ be an upper-tail potential with representing measure $\mu_g$.  Then, in the sense of distributions on $\Rp^d$,
\begin{equation}\label{eq:mixed-derivative-measure}
 \mu_g=(-1)^dD_1\cdots D_dg.
\end{equation}
In particular, the representing measure is uniquely determined by $g$.
\end{lemma}

\begin{proof}
By Lemma~\ref{lem:integrated-variation}, $g\in L^1_{\mathrm{loc}}(\Rp^d)$.  Let $\varphi\in C_c^\infty(\Rp^d)$.  Fubini's theorem is applicable because the support of $\varphi$ is contained in some rectangle $\Q(R)$ and the resulting absolute integral is bounded by a constant multiple of $\mathcal M_g(R)$.  Using \eqref{eq:tail-potential}, we obtain
\begin{align*}
 \langle D_1\cdots D_dg,\varphi\rangle
 &=(-1)^d\int_{\Rp^d}g(x)\,\partial_1\cdots\partial_d\varphi(x)\,dx\\
 &=(-1)^d\int_{\Rp^d}
 \left[\int_{\Q(t)}\partial_1\cdots\partial_d\varphi(x)\,dx\right]d\mu_g(t)\\
 &=(-1)^d\int_{\Rp^d}\varphi(t)\,d\mu_g(t).
\end{align*}
In the last step, the fundamental theorem of calculus is applied successively in every coordinate; all lower-boundary terms vanish because $\varphi$ has compact support in the open orthant.  This proves \eqref{eq:mixed-derivative-measure}.  Uniqueness follows because two Radon measures which define the same distribution agree.
\end{proof}

\begin{defi}[Rectangular mixed general monotonicity]\label{def:rmgm}
Let $K\ge1$ and $\lambda\in[1,\infty)^d$.  An upper-tail potential belongs to
$\mathrm{RMGM}(K,\lambda)$ if
\begin{equation}\label{eq:rmgm}
 \mathcal M_g(r)
 \le K\int_{\Q(\lambda r)}|g(x)|\,dx,
 \qquad r\in\Rp^d.
\end{equation}
\end{defi}

By Lemma~\ref{lem:integrated-variation}, condition \eqref{eq:rmgm} says that the local average of the mixed total-variation tail is controlled by a slightly enlarged local average of the function.  It is an integrated no-cancellation condition for the full mixed derivative.  Unlike a coordinatewise monotonicity assumption, it permits controlled cancellation in that derivative.

\begin{rem}[Separate monotonicity and mixed derivatives]\label{rem:separate-monotonicity}
An upper-tail potential of a positive measure is nonincreasing in every coordinate.
Separate monotonicity alone does not imply that the full mixed derivative is a Radon
measure.  For example, on $\Rp^2$ put
\[
 h(x_1,x_2)=\mathbf1_{\{x_1+x_2<1\}}.
\]
This function is nonnegative and separately nonincreasing.  In the coordinates
$u=x_1+x_2$, $v=x_1-x_2$, however, $D_1D_2h$ is the first distributional
derivative of a measure on $u=1$.  More explicitly, choose
$\chi\in C_c^\infty(-1,1)$ with $\chi'(0)\ne0$ and
$\eta\in C_c^\infty(-1/4,1/4)$ with $\int\eta\ne0$.  For $n\ge4$ the test functions
$\varphi_n(x)=\chi(n(u-1))\eta(v)$ have uniformly bounded supremum norms and
supports in a fixed compact subset of $\Rp^2$, while
\[
 \langle D_1D_2h,\varphi_n\rangle
 =\frac n2\chi'(0)\int_{\mathbb R}\eta(v)\,dv.
\]
This is incompatible with the total-variation bound for a Radon measure on that
compact set.  Thus $h$ is not an upper-tail potential in
Definition~\ref{def:tail-potential}.  The mixed-tail hypothesis is a separate
structural assumption, and its positive-measure subclass is a proper subset of
the coordinatewise monotone cone.
\end{rem}

For an upper-tail potential define its coordinatewise even extension away from the spatial coordinate hyperplanes by
\begin{equation}\label{eq:even-extension}
 f(x)=g(|x_1|,\ldots,|x_d|)
 \qquad\text{when }x_1\cdots x_d\ne0.
\end{equation}
On the spatial coordinate hyperplanes, choose arbitrary measurable values, for instance zero.  Whenever $g$ extends continuously or smoothly to the closed orthant, we instead use the corresponding even boundary values so that this regularity is preserved.  All such choices determine the same almost-everywhere class and hence the same Fourier integrals and mixed norms.

The ordinary Fourier integral need not be absolutely convergent.  We therefore use Abel summation.

\begin{defi}[Abel Fourier transform]\label{def:abel-transform}
For $\varepsilon\in(0,\infty)^d$ put
\[
 \widehat f_{\varepsilon}(\xi)
 =\int_{\R^d}f(x)e^{-\sum_{k=1}^d\varepsilon_k|x_k|}e^{i x\cdot\xi}\,dx.
\]
For frequencies with no zero coordinate, define
\[
 \widehat f_A(\xi)=\lim_{\varepsilon\to0+}\widehat f_{\varepsilon}(\xi),
\]
provided the limit exists.  Here $\varepsilon\to0+$ means that $\varepsilon_k\downarrow0$ for every $k$, along an arbitrary path in $(0,\infty)^d$; thus existence includes independence of that path.  On the frequency coordinate hyperplanes we set $\widehat f_A(\xi)=0$.  Whenever a product weight is used, we also define the weighted product $\Pi_\alpha(\xi)\widehat f_A(\xi)$ to be zero there, including when some $\alpha_k<0$.  These conventions do not affect any mixed-norm estimate.
\end{defi}

\begin{prop}[Mixed integration formula]\label{prop:mixed-transform-formula}
Let $g$ be an upper-tail potential and let $f$ be its even extension.  Then the Abel transform exists for every $\xi$ with no zero coordinate and
\begin{equation}\label{eq:mixed-transform-formula}
 \widehat f_A(\xi)
 =2^d\int_{\Rp^d}
 \prod_{k=1}^d\frac{\sin(t_k\xi_k)}{\xi_k}\,d\mu_g(t).
\end{equation}
Consequently,
\begin{equation}\label{eq:pointwise-mixed-bound}
 |\widehat f_A(\xi)|
 \le2^d\mathcal M_g(|\xi|^{-1}),
\end{equation}
where $|\xi|^{-1}=(|\xi_1|^{-1},\ldots,|\xi_d|^{-1})$.
The right-hand side of \eqref{eq:mixed-transform-formula}, viewed as a function of $\xi$, is continuous on every compact subset of $\{\xi\in\R^d:\xi_1\cdots\xi_d\ne0\}$; in particular, $\widehat f_A$ is measurable.  If $f\in L^1(\R^d)$, then $\widehat f_A=\widehat f$ almost everywhere.
\end{prop}

\begin{proof}
Evenness gives
\[
 \widehat f_{\varepsilon}(\xi)
 =2^d\int_{\Rp^d}g(x)e^{-\varepsilon\cdot x}
 \prod_{k=1}^d\cos(x_k\xi_k)\,dx.
\]
The integral is absolutely convergent because, by \eqref{eq:tail-potential} and Fubini,
\[
 \int_{\Rp^d}e^{-\varepsilon\cdot x}|g(x)|\,dx
 \le\int_{\Rp^d}\prod_{k=1}^d
 \left(\int_0^{t_k}e^{-\varepsilon_kx_k}\,dx_k\right)d|\mu_g|(t)
 \le\mathcal M_g(\varepsilon^{-1}).
\]
A second application of Fubini yields
\begin{equation}\label{eq:epsilon-formula}
 \widehat f_{\varepsilon}(\xi)
 =2^d\int_{\Rp^d}\prod_{k=1}^d
 \Phi_{\varepsilon_k,\xi_k}(t_k)\,d\mu_g(t),
\end{equation}
where
\[
 \Phi_{\varepsilon,\xi}(t)=\int_0^t e^{-\varepsilon x}\cos(\xi x)\,dx.
\]
For $\xi\ne0$,
\[
 |\Phi_{\varepsilon,\xi}(t)|
 \le \min\left\{t,\frac{2}{|\xi|}\right\}
 \le2\min\{t,|\xi|^{-1}\}.
\]
Thus the integrand in \eqref{eq:epsilon-formula} is dominated by an integrable function, in view of \eqref{eq:mixed-moment-finite}.  Since
$\Phi_{\varepsilon,\xi}(t)\to\sin(t\xi)/\xi$, dominated convergence proves \eqref{eq:mixed-transform-formula}; the same majorant works for every path along which all $\varepsilon_k\downarrow0$, so the limit is path independent.  The inequality $|\sin u|\le\min\{|u|,1\}$ gives \eqref{eq:pointwise-mixed-bound}.

Let $E$ be a compact subset of $\{\xi:\xi_1\cdots\xi_d\ne0\}$.  There are numbers $\delta_k>0$ such that $|\xi_k|\ge\delta_k$ on $E$.  The integrand in \eqref{eq:mixed-transform-formula} is then bounded uniformly for $\xi\in E$ by
\[
 \prod_{k=1}^d\min\{t_k,\delta_k^{-1}\},
\]
which is $|\mu_g|$-integrable by \eqref{eq:mixed-moment-finite}.  A further application of dominated convergence proves continuity on $E$.  Finally, if $f\in L^1$, ordinary dominated convergence in the defining Fourier integral shows that Abel summation agrees with $\widehat f$ away from the coordinate hyperplanes, and hence almost everywhere.
\end{proof}

\section{The dilation-balanced anisotropic Pitt theorem}\label{sec:main}

\begin{thm}[Mixed-norm product Pitt inequality]\label{thm:main}
Let
\[
 1\le p_k\le q_k<\infty,
 \qquad
 -\frac1{q_k}<\alpha_k<1-\frac1{q_k},
 \qquad k=1,\ldots,d,
\]
and define
\begin{equation}\label{eq:main-beta}
 \beta_k=1-\frac1{p_k}-\frac1{q_k}-\alpha_k.
\end{equation}
Let $g\in\mathrm{RMGM}(K,\lambda)$, and let $f$ be its coordinatewise even extension.  If
$\Pi_\beta f\in L_{\vec p}(\R^d)$, then
\begin{equation}\label{eq:main-pitt}
 \|\Pi_\alpha\widehat f_A\|_{L_{\vec q}(\R^d)}
 \le C
 \|\Pi_\beta f\|_{L_{\vec p}(\R^d)}.
\end{equation}
The constant depends only on $d,\vec p,\vec q,\alpha,K$, and $\lambda$.  If $f\in L^1(\R^d)$, the Abel transform in \eqref{eq:main-pitt} may be replaced by the ordinary Fourier transform.
\end{thm}

\begin{proof}
By Proposition~\ref{prop:mixed-transform-formula} and \eqref{eq:rmgm}, for positive coordinate frequencies,
\[
 |\widehat f_A(\xi)|
 \le2^dK\int_{\Q(\lambda/\xi)}|g(x)|\,dx.
\]
Proposition~\ref{prop:mixed-hardy} therefore gives
\[
 \|\Pi_\alpha\widehat f_A\|_{L_{\vec q}(\Rp^d)}
 \le C\|\Pi_\beta g\|_{L_{\vec p}(\Rp^d)}.
\]
Both $f$ and $\widehat f_A$ are even in each coordinate.  Folding the mixed norms over the $2^d$ orthants changes only the constant and proves \eqref{eq:main-pitt}.
\end{proof}

\begin{cor}[Equal-exponent formulation]\label{cor:equal-exponents}
Let $1\le p\le q<\infty$, let $\alpha\in\R^d$ satisfy
$-1/q<\alpha_k<1-1/q$ for every $k$, and put
\[
 \beta_k=1-\frac1p-\frac1q-\alpha_k.
\]
Then every coordinatewise even extension $f$ of a function in $\mathrm{RMGM}(K,\lambda)$ for which $\Pi_\beta f\in L^p(\R^d)$ satisfies
\[
 \left(\int_{\R^d}\Pi_\alpha(\xi)^q|\widehat f_A(\xi)|^q\,d\xi\right)^{1/q}
 \le C
 \left(\int_{\R^d}\Pi_\beta(x)^p|f(x)|^p\,dx\right)^{1/p}.
\]
\end{cor}

The balance in Theorem~\ref{thm:main} is not an artifact of the proof.

\begin{prop}[Coordinatewise scaling necessity]\label{prop:scaling}
The class $\mathrm{RMGM}(K,\lambda)$ is invariant under independent coordinate dilations.  Consequently, if an estimate of the form
\[
 \|\Pi_\alpha\widehat f_A\|_{L_{\vec q}}
 \le C\|\Pi_\rho f\|_{L_{\vec p}}
\]
holds uniformly for all coordinate dilates of one admissible function for which both norms are finite and nonzero, then necessarily
\[
 \rho_k=1-\frac1{p_k}-\frac1{q_k}-\alpha_k
 \quad\text{for every }k.
\]
\end{prop}

\begin{proof}
For $s\in(0,\infty)^d$, set $g_s(x)=g(s_1x_1,\ldots,s_dx_d)$.  The representing measure is the push-forward satisfying $\mu_{g_s}(E)=\mu_g(sE)$.  Direct changes of variables give
\[
 \mathcal M_{g_s}(r)
 =\left(\prod_{k=1}^d s_k^{-1}\right)\mathcal M_g(sr)
\]
and
\[
 \int_{\Q(\lambda r)}|g_s(x)|\,dx
 =\left(\prod_{k=1}^d s_k^{-1}\right)
 \int_{\Q(\lambda sr)}|g(y)|\,dy.
\]
Thus \eqref{eq:rmgm} is preserved with the same constants.

The Fourier transform scales as
\[
 \widehat{f_s}_A(\xi)=\left(\prod_{k=1}^ds_k^{-1}\right)
 \widehat f_A(\xi_1/s_1,\ldots,\xi_d/s_d).
\]
The left-hand norm therefore gains the factor
$\prod_ks_k^{\alpha_k-1+1/q_k}$, whereas the right-hand norm gains
$\prod_ks_k^{-\rho_k-1/p_k}$.  Since the $s_k$ are independent, equality of the exponents is necessary.
\end{proof}

\begin{lemma}[Nontriviality of the exponent range]\label{lem:upper-range}
Let $g\in\mathrm{RMGM}(K,\lambda)$ be nonzero, let $1\le p_k<\infty$, and let
$p_k'$ be the conjugate exponent, with $p_k'=\infty$ when $p_k=1$.
Suppose $\beta_k<1/p_k'$ for every $k$, where $1/\infty=0$.
If $\|\Pi_\beta g\|_{L_{\vec p}(\Rp^d)}<\infty$, then
$\beta_k>-1/p_k$ for every $k$.
Under the balance \eqref{eq:main-beta}, these two bounds on $\beta_k$ are
equivalent to $-1/q_k<\alpha_k<1-1/q_k$.
\end{lemma}
\begin{proof}
Since $g\ne0$, Lemma~\ref{lem:intrinsic-mixed-derivative} gives $\mu_g\ne0$; fix
$r^0\in\Rp^d$ with $m_0:=|\mu_g|(\Qt(r^0))>0$. For $r\le r^0$ componentwise,
$\mathcal M_g(r)\ge\big(\prod_k r_k\big)|\mu_g|(\Qt(r))\ge m_0\prod_k r_k$, so by
\eqref{eq:rmgm}
\[
 \int_{\Q(\lambda r)}|g|\ \ge\ K^{-1}m_0\prod_{k=1}^d r_k .
\]
On the other hand, H\"older's inequality for mixed norms on $B=\Q(\lambda r)$ gives
\[
 \int_{B}|g|\le\|\Pi_\beta g\|_{L_{\vec p}(B)}\,\|\Pi_{-\beta}\|_{L_{\vec p'}(B)}
 =\|\Pi_\beta g\|_{L_{\vec p}(B)}
 \prod_{k=1}^d c_k(\lambda_kr_k)^{1/p_k'-\beta_k},
\]
where
\[
 c_k=
 \begin{cases}
 (1-\beta_kp_k')^{-1/p_k'},&p_k>1,\\
 1,&p_k=1.
 \end{cases}
\]
For $p_k=1$ the dual factor is an essential supremum, finite because
$\beta_k<0$.  Hence, for all $r\le r^0$,
\[
 \|\Pi_\beta g\|_{L_{\vec p}(\Q(\lambda r))}\ \ge\ c\prod_{k=1}^d r_k^{\,1/p_k+\beta_k}.
\]
Fix $k$ and let $r_k\downarrow0$ with $r_j=r^0_j$ for $j\ne k$. If $p_k\beta_k\le-1$,
the right-hand side stays bounded below by a positive constant, whereas the left-hand
side tends to $0$ by dominated convergence applied successively in each variable of
the mixed norm, because $\Pi_\beta g\in L_{\vec p}(\Rp^d)$. This contradiction proves
the lemma.
\end{proof}

\section{Intrinsic subclasses and nonseparable examples}\label{sec:examples}

We first give a broad sufficient condition directly in terms of the polar decomposition of the mixed derivative measure.

\begin{prop}[Sectorial mixed derivatives]\label{prop:sectorial}
Let $g$ be an upper-tail potential.  Suppose there are a complex number $\omega$ with $|\omega|=1$ and a number $\kappa\in(0,1]$ such that
\begin{equation}\label{eq:sectorial}
 \operatorname{Re}\left(\omega\frac{d\mu_g}{d|\mu_g|}(t)\right)\ge\kappa
 \quad\text{for }|\mu_g|\text{-almost every }t.
\end{equation}
Then $g\in\mathrm{RMGM}(\kappa^{-1},\one)$.
In particular, every upper-tail potential of a positive measure belongs to $\mathrm{RMGM}(1,\one)$.
\end{prop}

\begin{proof}
For every $x\in\Rp^d$,
\[
 |g(x)|\ge\operatorname{Re}(\omega g(x))
 =\int_{\Qt(x)}
 \operatorname{Re}\left(\omega\frac{d\mu_g}{d|\mu_g|}(t)\right)d|\mu_g|(t)
 \ge\kappa|\mu_g|(\Qt(x)).
\]
Integrating over $\Q(r)$ and using Lemma~\ref{lem:integrated-variation} proves the claim.
\end{proof}

In particular, if a nonzero positive Radon measure satisfying
\eqref{eq:mixed-moment-finite} is not a tensor product of one-dimensional
Radon measures, then its upper-tail potential is nonseparable.  Indeed, a
tensor-product representation of the potential would, by
Lemma~\ref{lem:tensor-product-derivative} below and
\eqref{eq:mixed-derivative-measure}, force its representing measure to be a
tensor product.

\begin{lemma}[Mixed derivatives of tensor products]\label{lem:tensor-product-derivative}
Let $g_k\in L^1_{\mathrm{loc}}(0,\infty)$ be nonzero and set
\[
 g(x)=\prod_{k=1}^d g_k(x_k).
\]
Then, as distributions on $\Rp^d$,
\begin{equation}\label{eq:tensor-product-derivative}
 D_1\cdots D_dg=Dg_1\otimes\cdots\otimes Dg_d.
\end{equation}
If the distribution in \eqref{eq:tensor-product-derivative} is a nonzero Radon measure, then every $Dg_k$ is a nonzero Radon measure.  If, in addition, that product measure is absolutely continuous with respect to Lebesgue measure, then every $Dg_k$ is absolutely continuous with respect to one-dimensional Lebesgue measure.  In that case, writing
$d(Dg_k)=v_k(x_k)\,dx_k$, the density of the product measure is
$\prod_{k=1}^d v_k(x_k)$ almost everywhere.
\end{lemma}

\begin{proof}
Identity \eqref{eq:tensor-product-derivative} follows first on elementary test functions $\varphi(x)=\prod_k\varphi_k(x_k)$ from
\[
 \langle D_1\cdots D_dg,\varphi\rangle
 =\prod_{k=1}^d\langle Dg_k,\varphi_k\rangle,
\]
and then on all of $C_c^\infty(\Rp^d)$ by the standard tensor-product characterization of distributions.

Write $T_k=Dg_k$ and
$T=T_1\otimes\cdots\otimes T_d$.  Suppose first that $T$ is a nonzero
Radon measure.  Then none of the $T_k$ is zero.  Fix $k$, choose
$\psi_j\in C_c^\infty(0,\infty)$ such that
$c_j:=\langle T_j,\psi_j\rangle\ne0$ for every $j\ne k$, and put
$C_k=\prod_{j\ne k}c_j$.  For every
$\varphi\in C_c^\infty(0,\infty)$,
\begin{equation}\label{eq:tensor-factor-contraction}
 \langle T_k,\varphi\rangle
 =\frac{1}{C_k}
 \left\langle T,
 \varphi(x_k)\prod_{j\ne k}\psi_j(x_j)\right\rangle.
\end{equation}
If $\varphi$ is supported in a fixed compact interval $J$, the total-variation
bound for $T$ on
\[
 \prod_{j<k}\supp\psi_j\times J\times
 \prod_{j>k}\supp\psi_j
\]
shows that the right-hand side of
\eqref{eq:tensor-factor-contraction} is bounded by a constant times
$\|\varphi\|_\infty$.  Thus $T_k$ is a distribution of order zero and hence
a Radon measure.  It is nonzero by construction.

Now assume in addition that $T$ is absolutely continuous, say
$dT=m(x)\,dx$ with $m\in L^1_{\mathrm{loc}}(\Rp^d)$.  With the same test
functions, define
\[
 v_k(x_k)=\frac{1}{C_k}
 \int_{\Rp^{d-1}}m(x)\prod_{j\ne k}\psi_j(x_j)\,dx_{\widehat k},
 \qquad
 dx_{\widehat k}:=\prod_{j\ne k}dx_j.
\]
Fubini's theorem gives $v_k\in L^1_{\mathrm{loc}}(0,\infty)$, and
\eqref{eq:tensor-factor-contraction} becomes
\[
 \langle T_k,\varphi\rangle
 =\int_0^\infty \varphi(x_k)v_k(x_k)\,dx_k.
\]
Hence $T_k=v_k(x_k)\,dx_k$ is absolutely continuous.  Repeating the argument
for every $k$, the tensor product of the factor measures has density
$\prod_{k=1}^d v_k(x_k)$.  Since this tensor product is $T$, uniqueness of
Radon--Nikodym densities yields
$m(x)=\prod_{k=1}^d v_k(x_k)$ almost everywhere.
\end{proof}

\begin{lemma}[Support of a tensor-product distribution]\label{lem:tensor-product-support}
Let $U_k\subset\R$ be open and let $T_k\in\mathcal D'(U_k)$ be nonzero for
$k=1,\ldots,d$.  Then
\begin{equation}\label{eq:tensor-product-support}
 \supp(T_1\otimes\cdots\otimes T_d)
 =\supp T_1\times\cdots\times\supp T_d.
\end{equation}
\end{lemma}

\begin{proof}
If $x_k\notin\supp T_k$ for some $k$, there is a neighborhood $V_k$ of
$x_k$ on which $T_k$ vanishes.  Hence the tensor product vanishes on every
product neighborhood having $V_k$ as its $k$th factor, which proves the
inclusion from left to right in \eqref{eq:tensor-product-support}.

Conversely, let $x_k\in\supp T_k$ for every $k$ and let $V$ be any
neighborhood of $x=(x_1,\ldots,x_d)$.  Choose product neighborhoods
$V_1\times\cdots\times V_d\subset V$ with $x_k\in V_k$.  By the definition
of distributional support, there are $\varphi_k\in C_c^\infty(V_k)$ such
that $\langle T_k,\varphi_k\rangle\ne0$.  Therefore
\[
 \left\langle T_1\otimes\cdots\otimes T_d,
 \prod_{k=1}^d\varphi_k\right\rangle
 =\prod_{k=1}^d\langle T_k,\varphi_k\rangle\ne0.
\]
Thus the tensor product does not vanish on $V$, so
$x\in\supp(T_1\otimes\cdots\otimes T_d)$.
\end{proof}

\begin{ex}[An infinite-dimensional cone with nonseparable elements]\label{ex:positive-density}
Let $w\ge0$ be any locally integrable function on $\Rp^d$ such that
\[
 \int_{\Rp^d}\prod_{k=1}^d\min\{r_k,t_k\}w(t)\,dt<\infty
 \quad\text{for every }r\in\Rp^d,
\]
and define
\begin{equation}\label{eq:density-tail}
 g_w(x)=\int_{\Qt(x)}w(t)\,dt.
\end{equation}
Then $\mu_{g_w}=w(t)dt$ and $g_w\in\mathrm{RMGM}(1,\one)$.  The density $w$ is arbitrary and need not factor into one-dimensional functions.  If $w$ is not equal almost everywhere to a product of one-dimensional densities, Lemma~\ref{lem:tensor-product-derivative}, together with \eqref{eq:mixed-derivative-measure}, shows that $g_w$ is not a tensor product.  Thus the class contains an infinite-dimensional collection of genuinely nonseparable functions.

For example, if $w\in C_c^\infty(\Rp^d)$ is nonnegative and non-product, then $g_w$ extends smoothly to the closed orthant and its coordinatewise even extension is compactly supported and nonseparable.  Another explicit choice is
\[
 w(t)=\mathbf1_{\{t_1+\cdots+t_d<1\}},
 \qquad
 g_w(x)=\frac{(1-x_1-\cdots-x_d)_+^d}{d!}.
\]
\end{ex}

The class is not confined to the coordinatewise monotone cone.

\begin{prop}[A real nonmonotone nonseparable family]\label{prop:atomic-family}
Let $d\ge2$, let $a,b\in\Rp^d$ satisfy $a_k<b_k$ for every $k$, and let $0<c<1$.  Put
\begin{equation}\label{eq:atomic-measure}
 \mu=\delta_b-c\delta_a,
 \qquad
 g_{a,b,c}(x)
 =\mathbf1_{\{0<x_k\le b_k\,\text{ for all }k\}}(x)
 -c\mathbf1_{\{0<x_k\le a_k\,\text{ for all }k\}}(x).
\end{equation}
Then
\[
 g_{a,b,c}\in\mathrm{RMGM}\left(\frac{1+c}{1-c},\one\right).
\]
The function is nonnegative but is not nonincreasing in any coordinate for which $a_k<b_k$.  It is also not a tensor product when at least two coordinate inequalities are strict.
\end{prop}

\begin{proof}
The upper-tail potential of $\delta_b-c\delta_a$ is exactly \eqref{eq:atomic-measure}.  For $r\in\Rp^d$, write
\[
 A(r)=\prod_{k=1}^d\min\{r_k,b_k\},
 \qquad
 B(r)=\prod_{k=1}^d\min\{r_k,a_k\}.
\]
Since $0\le B(r)\le A(r)$,
\[
 \mathcal M_{g_{a,b,c}}(r)=A(r)+cB(r),
 \qquad
 \int_{\Q(r)}|g_{a,b,c}(x)|\,dx=A(r)-cB(r).
\]
Therefore
\[
 \frac{\mathcal M_{g_{a,b,c}}(r)}{\int_{\Q(r)}|g_{a,b,c}|}
 =\frac{A(r)+cB(r)}{A(r)-cB(r)}
 \le\frac{1+c}{1-c}.
\]

Fix all coordinates except $x_k$ below the corresponding components of $a$.  As $x_k$ crosses $a_k$, the value of $g_{a,b,c}$ increases from $1-c$ to $1$, proving failure of coordinatewise monotonicity.  To see nonseparability, choose two strict coordinates $i\ne j$, intervals $I_i^-\subset(0,a_i)$ and $I_i^+\subset(a_i,b_i)$, intervals $I_j^-\subset(0,a_j)$ and $I_j^+\subset(a_j,b_j)$, and, in every remaining coordinate, an interval contained in $(0,a_k)$.  On the resulting four product rectangles the values are respectively $1-c,1,1,1$.  If $g_{a,b,c}$ were a tensor product almost everywhere, Fubini's theorem would imply for almost every associated quadruple the rank-one identity
\[
 g(x_i^-,x_j^-)g(x_i^+,x_j^+)
 =g(x_i^-,x_j^+)g(x_i^+,x_j^-),
\]
with the remaining coordinates fixed at common admissible values.  This would give $(1-c)\cdot1=1\cdot1$, a contradiction.
\end{proof}

\begin{cor}[Smooth nonmonotone nonseparable examples]\label{cor:smooth-family}
Let $d\ge2$, let $0<c<1$, and choose a nonzero product bump
\[
 \phi(z)=\prod_{k=1}^d\phi_k(z_k),
 \qquad
 \phi_k\in C_c^\infty((0,\varepsilon)),
 \qquad \phi_k\ge0.
\]
Let $a,b\in\Rp^d$ satisfy $b_k-a_k>\varepsilon$ for every $k$.  Denote by $\nu_v$ the measure with density $\phi(t-v)$, extended by zero outside $v+(0,\varepsilon)^d$, and put
\[
 \mu=\nu_b-c\nu_a,
 \qquad
 g(x)=\mu(\Qt(x)).
\]
Then
\[
 g\in\mathrm{RMGM}\left(\frac{1+c}{1-c},\one\right).
\]
Moreover, $g$ extends smoothly to the closed orthant, its even extension belongs to $C_c^\infty(\R^d)$, and $g$ is neither coordinatewise nonincreasing nor a tensor product.
\end{cor}

\begin{proof}
Write $G_v(x)=\nu_v(\Qt(x))$ and
\[
 M_v(r)=\int_{\Rp^d}\prod_{k=1}^d\min\{r_k,t_k\}\,d\nu_v(t).
\]
Because $b>a$ componentwise, $G_a\le G_b$ and $M_a\le M_b$.  The two translated supports are disjoint, so $|\mu|=\nu_b+c\nu_a$.  Consequently $g=G_b-cG_a\ge0$ and
\[
 \mathcal M_g(r)=M_b(r)+cM_a(r),
 \qquad
 \int_{\Q(r)}g(x)\,dx=M_b(r)-cM_a(r).
\]
The same ratio estimate as in Proposition~\ref{prop:atomic-family} proves the RMGM bound.

The supports are separated from the coordinate hyperplanes, so $g$ is independent of $x_k$ for sufficiently small $x_k$, for each $k$; hence its even extension is smooth.  If all coordinates except $x_k$ are fixed below the support of $\nu_a$, then $G_b$ remains constant while $G_a$ decreases as $x_k$ crosses the first translated bump.  Thus $g$ increases there.  Finally, the mixed derivative density is supported on the union of two separated Cartesian-product compact sets.  By Lemma~\ref{lem:tensor-product-support}, the support of a nonzero tensor-product distribution is the Cartesian product of the supports of its one-dimensional factors, which would also contain nonempty cross product sets.  Hence the mixed derivative, and therefore $g$, cannot be a tensor product.
\end{proof}

For the even extension of \eqref{eq:atomic-measure}, the ordinary Fourier transform is explicit:
\begin{equation}\label{eq:atomic-transform}
 \widehat f_{a,b,c}(\xi)
 =2^d\left[
 \prod_{k=1}^d\frac{\sin(b_k\xi_k)}{\xi_k}
 -c\prod_{k=1}^d\frac{\sin(a_k\xi_k)}{\xi_k}
 \right].
\end{equation}
For each $k$ and each $\ell\in\{a_k,b_k\}$, the quotient $\sin(\ell\xi_k)/\xi_k$ in \eqref{eq:atomic-transform} is interpreted continuously at $\xi_k=0$, with value $\ell$.  Thus, for every fixed $c\in(0,1)$, Theorem~\ref{thm:main} produces a dilation-balanced Pitt inequality which is uniform in $a$ and $b$ for this nonmonotone and nonseparable family; the constant depends on $c$ through the RMGM constant $(1+c)/(1-c)$.

\begin{cor}[Exact finite-norm range for compactly supported examples]\label{cor:sharp-range}
Let $g$ be one of the following: a positive-density example $g_w$ from Example~\ref{ex:positive-density} with $w$ nonzero, integrable, and compactly supported; a function from Proposition~\ref{prop:atomic-family}; or a function from Corollary~\ref{cor:smooth-family}.  Let $f$ be its even extension.  Then both norms in \eqref{eq:main-pitt} are finite exactly in the natural range
\begin{equation}\label{eq:natural-range}
 -\frac1{q_k}<\alpha_k<1-\frac1{q_k},
 \qquad k=1,\ldots,d,
\end{equation}
where $\beta$ is given by \eqref{eq:main-beta}.  At the lower endpoint the Fourier-side norm diverges, while at the upper endpoint the spatial weighted norm diverges.
\end{cor}

\begin{proof}
In all three cases the representing measure is finite and compactly supported.  Choose $R_k>0$ so that
$\supp\mu_g\subset\prod_{k=1}^d[0,R_k]$.  In the positive-density case, nontriviality of $w$ and the identity $\Rp^d=\bigcup_{n=1}^\infty[1/n,n]^d$ imply that some rectangle bounded away from the coordinate hyperplanes carries positive $w$-mass.  Proposition~\ref{prop:mixed-transform-formula} and \eqref{eq:pointwise-mixed-bound} give
\begin{align*}
 |\widehat f(\xi)|
 &\le 2^d|\mu_g|(\Rp^d)
 \prod_{k=1}^d\min\{R_k,|\xi_k|^{-1}\}\\
 &\le C_R\prod_{k=1}^d\min\{1,|\xi_k|^{-1}\}.
\end{align*}
Hence the Fourier-side mixed norm is finite whenever
$-1/q_k<\alpha_k<1-1/q_k$ for every $k$.  Conversely, each example is nonnegative, integrable, and nonzero, so $\widehat f$ is continuous and $\widehat f(0)>0$.  Therefore there are $c_0,\delta>0$ such that $|\widehat f(\xi)|\ge c_0$ on $(-\delta,\delta)^d$.  Restricting the mixed norm to this box and evaluating the resulting product of one-dimensional power integrals forces $\alpha_k>-1/q_k$ for every $k$.

Each function $f$ is bounded and compactly supported.  Moreover, there are $c_1,\delta_1>0$ such that $f(x)\ge c_1$ on $(-\delta_1,\delta_1)^d$: for $g_w$, choose $\delta_1>0$ so that $\int_{[\delta_1,\infty)^d}w(t)\,dt>0$; for the atomic and smooth families, take $\delta_1$ below the lower translated support.  It follows that the spatial weighted mixed norm is finite exactly when
$p_k\beta_k>-1$ for every $k$.  By \eqref{eq:main-beta}, this is equivalent to
$\alpha_k<1-1/q_k$.  At equality the corresponding one-dimensional integral diverges logarithmically.
\end{proof}

\begin{prop}[Sign-changing atomic potentials]\label{prop:sign-changing-family}
Keep $a$, $b$ and the definition \eqref{eq:atomic-measure} from
Proposition~\ref{prop:atomic-family}, and now take $c>1$. The profile
$g_{a,b,c}$ changes sign, is nonseparable, and belongs to
\[
 \mathrm{RMGM}\left(\frac{1+c}{c-1},\one\right).
\]
For $c=1$, the RMGM condition fails for every finite dilation vector.
\end{prop}

\begin{proof}
The profile equals $1-c$ on $\Q(a)$ and $1$ on $\Q(b)\setminus\Q(a)$
almost everywhere. On these same sets the variation tail $W_g$ equals
$1+c$ and $1$, respectively. Both functions vanish outside $\Q(b)$.
For $c>1$ this gives the pointwise inequality
$W_g\le (1+c)(c-1)^{-1}|g_{a,b,c}|$ almost everywhere. Integration
over $\Q(r)$ gives the RMGM bound. The four-rectangle test in
Proposition~\ref{prop:atomic-family} gives nonseparability since
$1-c\ne1$. At $c=1$, take $r$ so small that $\lambda r<a$
componentwise. The left-hand side of \eqref{eq:rmgm} is positive
and the right-hand side is zero.
\end{proof}

\section{Sharpness of the parameter ranges}\label{sec:sharpness}

We record explicitly what is sharp in Theorem~\ref{thm:main} and in the examples above.  There are three different restrictions, and they have different origins.

First, the relation
\[
        \beta_k=1-\frac1{p_k}-\frac1{q_k}-\alpha_k
\]
is forced by independent coordinate dilations.  This is Proposition~\ref{prop:scaling}.  The argument is insensitive to the particular form of the RMGM condition: once a class is stable under the dilations
\[
        g(x)\mapsto g(s_1x_1,\ldots,s_dx_d),
\]
any uniform estimate with product weights must have exactly this balance in each coordinate.  Thus the theorem does not merely produce a sufficient exponent relation; it produces the only possible coordinatewise balance compatible with dilation invariance.

Second, the lower restrictions
\[
        \alpha_k>-\frac1{q_k}
\]
are forced at the frequency origin for the natural compactly supported examples.  In Corollary~\ref{cor:sharp-range}, the even extensions are nonnegative, integrable, and not identically zero.  Hence their Fourier transforms are continuous and nonzero at the origin.  Restricting the mixed norm to a small product neighborhood of $0$ gives precisely the one-dimensional integrability condition
\[
        \int_0^\delta \xi_k^{\alpha_kq_k}\,d\xi_k<\infty,
\]
which is equivalent to $\alpha_k>-1/q_k$.  Failure in one coordinate already makes the full mixed norm divergent.
For these examples the spatial norm remains finite when one lowers that
$\alpha_k$ to or below $-1/q_k$ with the other exponents in range.

Third, the upper restrictions
\[
        \alpha_k<1-\frac1{q_k}
\]
are forced, for the same compact examples, by the spatial norm near the coordinate hyperplanes.  Since
\[
        \beta_k=1-\frac1{p_k}-\frac1{q_k}-\alpha_k,
\]
the condition $p_k\beta_k>-1$ is equivalent to $\alpha_k<1-1/q_k$.  The examples in Corollary~\ref{cor:sharp-range} are bounded below by a positive constant near the origin, so this condition is necessary and sufficient for the spatial weighted norm.  At the endpoint one obtains logarithmic divergence.

Lemma~\ref{lem:upper-range} establishes the upper restrictions for every nonzero
RMGM potential with finite input norm, under the lower restrictions
$\alpha_k>-1/q_k$.  Its proof uses the RMGM condition and mixed H\"older
inequality, and applies to complex and signed potentials as well as to positive
ones.  Consequently the nontrivial finite-input range of the theorem is
\[
        -\frac1{q_k}<\alpha_k<1-\frac1{q_k}
        \qquad (k=1,\ldots,d).
\]
The Hardy estimate itself needs only the lower restrictions, but within the
RMGM class an exponent at or above an upper endpoint leaves no nonzero
potential with finite input norm.  Corollary~\ref{cor:sharp-range} supplies
compactly supported models throughout the open range and identifies logarithmic
divergence at both endpoints.  Thus the dilation balance is necessary, and the
stated open range is the full range in which the theorem has nonzero finite-input
examples.

The restriction $p_k\le q_k$ is sharp for the Hardy mechanism used in the proof.  Indeed, in the logarithmic variables of Lemma~\ref{lem:power-hardy}, the one-dimensional operator is
\[
 S_\delta H(s)=\int_{-\infty}^s e^{-\delta(s-u)}H(u)\,du.
\]
For $N>1$, take $H_N=\mathbf1_{[0,N]}$.  If $1\le s\le N$, then
\[
 S_\delta H_N(s)
 \ge\int_{s-1}^s e^{-\delta(s-u)}\,du
 =\frac{1-e^{-\delta}}{\delta}.
\]
Consequently,
\[
 \|S_\delta H_N\|_{L^q(\R)}\gtrsim N^{1/q},
 \qquad
 \|H_N\|_{L^p(\R)}=N^{1/p}.
\]
A uniform $L^p\to L^q$ bound therefore forces $1/q\le1/p$, equivalently
$p\le q$.  Conversely, Lemma~\ref{lem:power-hardy} holds throughout the full
range $1\le p\le q<\infty$, including the endpoint $p=1$, and the mixed
estimate follows by tensorization.  Testing the mixed operator on product
functions reduces necessity in each coordinate to the preceding one-dimensional
argument.  Thus $1\le p_k\le q_k<\infty$ is the exact finite-exponent range
for the rectangular Hardy route.  Different
hypotheses or additional cancellation might lead to other estimates, but not
to the dilation-balanced theorem proved here by this mechanism.

\section{Relation with one-dimensional and radial general-monotone theory}\label{sec:one-dimensional-radial}

The construction is intended as a product analogue of the general-monotone Fourier theory in one dimension and of the radial general-monotone theory.  In one dimension, an upper-tail potential has the form
\[
        g(x)=\mu_g([x,\infty)),
        \qquad x>0,
\]
and the identity
\[
        \mu_g=-Dg
\]
is the usual distributional way of saying that the variation measure controls the decrease of $g$.  If $\mu_g$ is positive, then $g$ is nonnegative and nonincreasing.  If $\mu_g$ is sectorial or signed with controlled cancellation, the condition is a direct analogue of the general-monotone principle: monotonicity is replaced by control of variation on appropriate intervals.  The Abel formula in Proposition~\ref{prop:mixed-transform-formula} reduces in this case to the familiar sine-kernel representation
\[
        \widehat f_A(\xi)=2\int_0^\infty \frac{\sin(t\xi)}{\xi}\,d\mu_g(t),
\]
which is one of the standard mechanisms behind one-dimensional Boas--Pitt estimates for monotone or general-monotone data \cite{LT08,LT09,LT11,DGT13}.

The multiparameter theory differs from a mere iteration of the one-dimensional statement in one essential respect.  The representing measure $\mu_g$ need not be a product measure.  Thus the full mixed derivative may couple all coordinates.  The examples in Section~\ref{sec:examples} exploit exactly this point: positive densities may be genuinely nonseparable, and signed atomic or smooth examples may fail coordinatewise monotonicity.  Nevertheless, the full mixed derivative still supplies one cancellation factor in every coordinate, and the RMGM condition controls the corresponding upper-tail total variation by a rectangular average of $|g|$.  This is the product-space replacement for the interval variation control used in the one-dimensional general-monotone theory.

The relation with radial general-monotone results is parallel but not literal.  Radial Fourier inequalities are governed by balls, radial derivatives, and Bessel or Hankel kernels; see for example \cite{GLT11,DGT13}.  The present setting is governed by rectangles, full mixed derivatives, and product sine kernels.  Therefore Theorem~\ref{thm:main} is not a radial theorem in disguise and is not obtained by restricting a radial result to product variables.  Instead, it is an anisotropic counterpart of the same principle: a derivative or variation object adapted to the geometry gives the cancellation needed in the Fourier kernel, and a Hardy-type inequality converts the resulting variation bound into a weighted norm estimate.

Positive mixed-tail potentials are coordinatewise nonincreasing, so they provide
an overlap with the anisotropic monotone setting studied in \cite{Muk15,Muk17}.
Remark~\ref{rem:separate-monotonicity} specifies the stronger mixed-derivative
structure of this subclass.  The exponent relation follows from coordinate
dilations, and the contribution here is its validity under the intrinsic RMGM
condition, which allows sectorial complex derivatives, signed derivatives with
controlled cancellation, and nonseparable mixed derivative measures.  The
mixed-tail representation and the Hardy argument give a self-contained proof
for this class, including $p_k=1$.

The exterior-tail Vitali condition in \cite{Tok26signed} controls
mixed variation outside a given coordinate scale and applies to
functions without prescribed parity. Under that condition a signed
reverse estimate holds in the classical product-weight range.
Condition \eqref{eq:rmgm} uses origin-anchored rectangles, and
Theorem~\ref{thm:main} gives a forward estimate for even extensions
with independent input and output mixed exponents.

\section{Concluding remarks}

The full mixed derivative supplies exactly one cancellation factor in every coordinate.  Rectangular mixed general monotonicity converts its total variation into a local Hardy average, and the anisotropic mixed-norm Pitt inequality follows by tensorizing one-dimensional power Hardy estimates.  The proof therefore separates three independent ingredients: mixed cancellation, a scale-invariant variation condition, and product Hardy boundedness.

The class is intrinsic because its representing measure is the full mixed distributional derivative.  It contains nonseparable elements because the derivative measure may couple the coordinates, and it contains functions outside the coordinatewise monotone cone.  In this sense the result is a product analogue of the general-monotone extension principle: a mixed-variation condition supplies the control needed for the dilation-balanced estimate.  The coordinatewise balance \eqref{eq:main-beta} is necessary by dilation.  Lemma~\ref{lem:upper-range} gives the full nontrivial finite-input range \eqref{eq:natural-range}, and the compactly supported model families realize this range and demonstrate endpoint divergence.

Natural further questions include reverse inequalities, Lorentz refinements, and criteria for \eqref{eq:rmgm} in terms of lower-order Hardy--Krause variations.  Such developments would connect the present mixed-tail class more directly with the two-sided Boas theory and with the broader general-monotone program.

\section*{Data availability}
No data were used or generated in the research described in this article.

\end{document}